\documentclass[11pt,oneside,reqno]{article}

\usepackage{amsfonts}
\usepackage{amsmath}
\usepackage{bbm}
\usepackage{nccmath}
\usepackage{graphicx}

\usepackage{amsmath,amsthm,amssymb}
\usepackage{mathrsfs}

 \newcommand{\PP}{{\mathbb P}}
\newcommand{\EE}{{\mathbb E}}

\theoremstyle{plain}

\newtheorem{theorem}{Theorem}[section]

\newtheorem{lem}[theorem]{Lemma} 

\theoremstyle{definition}

\newtheorem{remark}[theorem]{Remark}

\begin{document}

\title{\sc\bf\large\MakeUppercase{A Bound on the Rate of Convergence in the Central Limit Theorem for Renewal Processes under second moment conditions}}
\author{\sc
Gesine Reinert and Ce Yang \thanks{Oxford University; {\tt reinert@stats.ox.ac.uk}}}
\maketitle

\begin{abstract}
A famous result in renewal theory is the Central Limit Theorem for renewal processes. As  in applications usually only observations from a finite time interval are available, a bound on the Kolmogorov distance to the normal distribution is desirable. Here we provide an explicit non-uniform bound for the Renewal Central Limit Theorem based on Stein's method and track the explicit values of the constants. For this bound the inter-arrival time distribution is required to have only a second moment. As an intermediate result of independent interest we obtain explicit bounds in a non-central Berry-Ess\'{e}n theorem under second moment conditions. 
\end{abstract}



\section{Introduction} 
Let $Z, Z_i, i=1, 2, \ldots$ be i.i.d. non-negative random variables  with positive mean $\mu$ and finite variance $\sigma^2$, and let
$$X_t = \max \{ n: \sum_{i=1}^n Z_i \le t\}.$$
Then $(X_t, t \ge 0)$ is a classical renewal process.

 Renewal processes are a cornerstone in applied probability and appear in a number of applications, see for example \cite{omey} and references therein. As in applications, time is finite, a quantification of the the rate of convergence to normal is desirable. Also note that $X_t$ only takes on values in $\{ 0, 1, \ldots\}$. 
 In \cite{englund} it is shown that when $\gamma:= \EE ( | Z - \mu|^3) < \infty$ then 
 \begin{equation}\label{englund}
 \sup_{n =0, 1,\ldots} \left| \PP (X_t  < n) - \Phi \left( \frac{( n\mu - t) \sqrt{\mu}}{\sigma \sqrt{t}}\right) \right| \le 4 \left( \frac{\gamma}{\sigma} \right)^3  
 \left( \frac{\sqrt{\mu}}{\sqrt{t}} \right)^\frac12  
 \end{equation}
 where $\Phi$ is the c.d.f. of the standard normal distribution.  Also in  \cite{englund}  a similar bound is indicated when $Z$ possesses moments of order $\alpha$ for some $2 < \alpha < 3$. Under the third moment assumption, this bound was generalised to the bivariate case in \cite{ahmad}, which in turn was generalised to a $k$-variate process in \cite{niculescu}. The result was extended in \cite{roginsky} to allow for non-identically distributed inter-arrival times $Z_i$, again under third moment assumptions. 
 In \cite{billingsley}, Theorem 17.3, a functional central limit theorem for the renewal process is shown. In particular, as $t \rightarrow \infty$, $X_t$ is asymptotically normally distributed with mean $\frac{t}{\mu}$ and variance $\frac{\sigma^2 t }{ \mu^3}$. Hence second moments suffice for the normal approximation. Unfortunately  \cite{billingsley} does not give a bound on the rate of convergence. 
 
 In this paper we provide a bound on the rate of convergence in the case that $Z$ has only second moments; this bound is  of the order $t^{-\frac{1}{2}}$.   
As an intermediate result we provide explicit constants for a non-uniform Berry-Esse\'{e}n theorem, quantifying Theorem 2.2 in \cite{chenshao} (also Theorem 8.1 in \cite{ChGoSh}). Our main tool is Stein's method.
 
 The paper is organised as follows. In Section \ref{notations} we introduce notation, we give  some bounds on the tail of the normal distribution, and we provide some background from Stein's method. Section \ref{main} gives the main result, with a proof.
 The proof is based on the approach to obtain  non-uniform bound from sums of i.i.d. random variables in Chapter 8 of \cite{ChGoSh}, while deriving explicit bounds for the required intermediary results from that chapter.  Proofs of auxiliary results are given in Section \ref{proofs}. For convenience, in the Appendix we re-state results from \cite{ChGoSh} which are used in this paper.


\section{Notations, tail bounds, and results from Stein's method}\label{notations}

\subsection{Notations}

Let $Z_n$, $n\geq0$, be independent identically distributed, positive random variables. Let $T_n = Z_0 + ... + Z_{n-1}, n\geq1$. The process $X=(X_t, t\geq0)$  defined by $X_t = \#\left\{n \geq1 : T_n \leq t\right\}$ is the renewal process of interest.

For a renewal process $X_t$ whose inter-arrival times $Z_i$  have mean $\mu$ and variance $\sigma^2$, and $n, t \in \{0, 1, \ldots, \}$ fixed, we aim to compare $\PP( X_t \le n) = \PP\left( \frac{X_t - \frac{t}{\mu}}{\sigma \sqrt{t} \mu^{-\frac32}} \le \frac{(n \mu - t)\sqrt{\mu}}{\sigma \sqrt{t}} \right) $ to $\Phi\left(\frac{(n \mu - t)\sqrt{\mu}}{\sigma \sqrt{t}}\right)$. 

\subsection{Normal tail bounds}

The following results will be useful when we develop the bounds. Firstly, for
 every $w>0$, the standard normal tail bound
\begin{align}
\frac{1}{4(1+w^2)}e^{-\frac{w^2}{2}}\leq \Phi(-w)= 1-\Phi(w) \leq \min\left(\frac{1}{2},\frac{1} {w\sqrt{2\pi}} \right)e^{-\frac{w^2}{2}}  \label{4normaltail}    
\end{align}
holds.
This is a well-known result, see for example Inequality (2.11) and p.243 in \cite{ChGoSh}.  The next result assesses the smoothness of the standard normal c.d.f., as follows. 

\begin{lem}\label{lemma2.4}
For $\mu>0$, $\sigma>0$, $n \ge 1$ and $t>0$,  let
$$ I = 
 \left|  \Phi \left( \frac{n \mu -t }{\sigma \sqrt{n}} \right) - \Phi \left( \frac{( n\mu - t) \sqrt{\mu}}{\sigma \sqrt{t}}\right)\right|
$$
Then 
\begin{equation*}
I
\leq \begin{cases} 
\frac{\sqrt{2}}{e \sqrt{\pi}} \frac{\sigma}{\sqrt{t \mu }} 
& \mbox{if}\;  \,t \le n \mu;\\ 
\frac{16 }{ e^2 \sqrt{2  \pi } } \frac{{t^2}\sigma^3 }{{n^{\frac{1}{2}} \mu^2}{(t - n \mu)^2(\sqrt{n \mu t} + t)} } 
& \mbox{if}\,\,  t>  n \mu.
\end{cases} 
\end{equation*}
\end{lem}

A proof of Lemma \ref{lemma2.4} is in Section \ref{proofs}. 

\subsection{Results from Stein's method}

Stein's method, origating from \cite{stein} is a powerful tool to assess distances between distributions.  The proof of the statements below can be found in \cite{ChGoSh}, pp.13--16.
Let $W$ be a random variable and suppose that the aim is to bound
$| \PP(W \le z) - \Phi (z)| $ for all real $z$. 
For fixed $z \in \mathbb{R}$, the unique bounded solution $f(w):=f_z(w)$ of the so-called {\it Stein equation}
\begin{equation}
    f'(w)-wf(w)=\mathbbm{1}({w \leq z}) -\Phi(z) \label{3steineqn}
\end{equation}
is given by 
\begin{equation}
    f_z(w) = \begin{cases}
\sqrt{2\pi}e^{w^2/2}\Phi(w)[1-\Phi(z)]  & \mbox{if} \;w \leq z;\\
\sqrt{2\pi}e^{w^2/2}\Phi(z)[1-\Phi(w)]  & \mbox{if} \;w>z.\\
           \end{cases} \label{3steinsoln}
\end{equation}
With this solution, 
$$ \PP(W \le z) - \Phi (z) = \EE \{ f'(W) - W f(W) \}$$
and the right-hand side depends only on the distribution of $W$ and can often be bounded using Taylor expansion. 
Moreover, for the solution $f_z$ of the Stein equation \eqref{3steineqn},  $wf_z(w)$ in an increasing function of $w$, and for all real $w$, 
\begin{align} 
|f'_z(w)| &\leq 1;  \label{3lemma43}\\
0 < f_z(w) &\leq \min\left(\frac{\sqrt{2\pi}}{4}, \frac{1}{|z|}\right). \label{3lemma45}
\end{align}

\section{A non-uniform bound for the Renewal Central Limit Theorem}\label{main}

Our main result  is Theorem \ref{theorem31}. As $\PP (X_t \le n) = 0$ for $n < 1$, we restrict attention to the regime that $n \ge 1$.

\begin{theorem}[\bf{Bound for the Renewal Central Limit Theorem Under Second Moment Assumptions}]\label{theorem31}
Let $X=(X_t, t\geq 0)$ be a renewal process whose inter-arrival times $Z_n$, $n\geq0$, have finite mean $\mu \in (0,\infty)$ and finite variance $\sigma^2 \in (0,\infty)$. 
Then for $n \ge 1$, 
\begin{eqnarray}
\lefteqn{\left| \PP (X_t  \le n) - \Phi \left( \frac{( n\mu - t) \sqrt{\mu}}{\sigma \sqrt{t}}\right) \right| } \nonumber \\
&\leq & \mathbbm{1}( t \le n \mu) \frac{\sqrt{2}}{e \sqrt{\pi}} \frac{\sigma}{\sqrt{t} \mu }   +  \mathbbm{1}( t >  n \mu) \frac{32 }{ e^2 \sqrt{2  \pi } } \frac{1}{\sqrt{t}
 } \left( \frac{\sigma^3  }{\mu^2\sqrt{t} } +  \frac{ \sigma }{ (224)^2 \sqrt{\mu}}\right)
\nonumber  \\
&& + \, 50,990   \left(1+\left|\frac{t- n\mu} {\sigma \sqrt{n}}\right|\right) ^{-2} .
 \label{55theorem7}
\end{eqnarray}
\end{theorem}

Before we prove this results, here are some remarks. 

\begin{remark}
\begin{enumerate} 
\item 
The explicit value of the constant in Theorem $3.1$ is large. This is because the calculation of the constant is not optimized. As a result, the bound is not informative for small values of $n$.
\item
The bound is  the order of $t^{-\frac{1}{2}}$. The bound deteriorates for $t$ close to the expectation $n \mu$. 
\item 
Theorem $3.1$ does not assume the existence of the finite third moments. It holds as long as the inter-arrival times have finite variance. This result enables us to assess the rate of convergence in the Central Limit Theorem for example for a renewal process whose inter-arrival times $Z_i $ follow a Pareto $\mbox{Pareto} \;(n, \alpha)$-distribution
with $\alpha \in [2,3)$ for $i\geq 1$. 
\end{enumerate} 
\end{remark} 

For the proof of Theorem \ref{theorem31}, recall that $T_n = \sum_{i=1}^n Z_i$ has mean $n \mu $ and variance $n \sigma^2$, and $\PP(X_t \le n) = \PP (T_n \ge t)$. Moreover, the standardised $T_n$ satisfies the Central Limit Theorem. We decompose
\begin{eqnarray}
{\PP (X_t  \le n) - \Phi \left( \frac{( n\mu - t) \sqrt{\mu}}{\sigma \sqrt{t}}\right)}
 &=&  \PP (T_n \ge t) - \left\{ 1 - \Phi \left( \frac{t - n \mu }{\sigma \sqrt{n}} \right) \right\}
\label{term1} \\
 &&+   \Phi \left( \frac{n \mu -t }{\sigma \sqrt{n}} \right) - \Phi \left( \frac{( n\mu - t) \sqrt{\mu}}{\sigma \sqrt{t}}\right).
\label{term2}
\end{eqnarray} 

We bound the terms \eqref{term1} and \eqref{term2} separately.  For \eqref{term2} we employ the tail bounds for the normal distribution from Lemma \ref{lemma2.4}.  For \eqref{term1} we derive non-uniform bounds using ideas from Chapter 8 in \cite{ChGoSh} - our Theorem \ref{theorem5.1} is a version of Theorem 8.1 in \cite{ChGoSh} but with the constants in the bound made explicit. This bound is  of interest in their own right and hence we give it as a theorem.

\begin{theorem}\label{theorem5.1} Let $\xi_1, \xi_2, \ldots \xi_n$ be i.i.d. random variables with mean $\mu$ and variance $\sigma^2$. Let $W$ denote their sum, $W=\sum_{i=1}^n \xi_i$. Let 
\begin{equation*}
\beta_2 = \sum_{i=1}^n \EE \xi_i^2 \mathbbm{1}({|\xi_i|> 1})   \;\; \mbox{and} \;\; \beta_3 = \sum_{i=1}^n \EE |\xi_i|^3 \mathbbm{1}({|\xi_i|\leq 1}). 
\end{equation*}
Then, for all $z\in \mathbb{R}$, 
\begin{equation}
|\PP(W \leq z)-\Phi(z)|\leq 2\sum_{i=1}^n \PP\left(|\xi_i|>\frac{1\vee |z|}{4}\right) + C_2(1+|z|)^{-2}(\beta_2+\beta_3), \label{5theorem3}
\end{equation}
where
\begin{equation}
C_2\leq \begin{cases}
15    & if\; \beta_2+\beta_3\geq 1;\\
37      & if\;\beta_2+\beta_3<1 \;and\; |z|\leq 2;\\
25431     & if\; \beta_2+\beta_3<1 \;and\; |z|>2. 
\end{cases} \label{52C2}
\end{equation}
\end{theorem}

The proof of Theorem \ref{theorem5.1} is found in Section \ref{proofs}. The proof of Theorem  \ref{theorem31} is now almost immediate.

\begin{proof}[Proof of Theorem  \ref{theorem31}]
First, Term \eqref{term2} is bounded directly in Lemma \ref{lemma2.4}.
The bound arising from \eqref{term1} is less than 1 only when 
$$\frac{| t - n \mu| }{\sigma \sqrt{n}} \ge \sqrt{50,990} -1.$$
Hence if $\frac{| t - n \mu| }{\sigma \sqrt{n}} \le 224$, the claim is trivially true. 
 So we apply Lemma \ref{lemma2.4} for $\frac{| t - n \mu| }{\sigma \sqrt{n}} > 224$,
 which turns the non-uniform bound for the regime $t > n \mu$ and $n \ge 1$ into a uniform bound for the regime $ \frac{ t - n \mu }{\sigma \sqrt{n}} > 224$, for which it holds that 
 $$\frac{t}{t - n \mu} \le 1 +  \frac{n \mu}{ 224 \sigma \sqrt{n}}$$
 so that 
 \begin{eqnarray*}
{\frac{16 }{ e^2 \sqrt{2  \pi } } \frac{t^2\sigma^3 }{\sqrt{n} \mu^2(t - n \mu)^2(\sqrt{n \mu t} + t)}  }
 & \le & \frac{32 }{ e^2 \sqrt{2  \pi } } \frac{1}{\sqrt{t
 } }\left( \frac{\sigma^3  }{\mu^2\sqrt{t} } +  \frac{ \sigma }{ (224)^2 \sqrt{\mu}}\right) .
 \end{eqnarray*} 
 This gives the first part of the bound.

   For Term \eqref{term1}, using Theorem \ref{theorem5.1} it remains to show that
   $$ 2\sum_{i=1}^n \PP\left(|\xi_i|>\frac{1\vee |z|}{4}\right) \le  128 \left(\frac{1}{1+|z|}\right)^2
   $$
   with $\xi = \frac{Z_i -  \mu}{\sigma \sqrt{n}}$ and then apply this inequalty to $z = \left| \frac{t - n \mu}{\sigma \sqrt{n}} \right| .$
Note that
$\frac{1+|z|}{2}\leq 1\vee |z|$. So, using Markov's inequality, 
\begin{align*}
\sum_{i=1}^n \PP\left(|\xi_i|>\frac{1\vee |z|}{4}\right) &\leq \sum_{i=1}^n \PP\left(|\xi_i|>\frac{1+|z|}{8}\right)\\
&\leq \left(\frac{8}{1+|z|}\right)^2 \sum_{i=1}^n \EE \xi_i^2 =  \left(\frac{8}{1+|z|}\right)^2 .
\end{align*}
Setting $C= 128 + 2 C_2$ gives the assertion. 
\end{proof} 

\begin{remark}
With the notation from Theorem \ref{theorem5.1}, under the same assumptions as for Theorem \ref{theorem31}, using Theorem 3.3 in \cite{chenshao2}  with $\xi_i =\frac{Z_i -  \mu}{\sigma \sqrt{n}} $ to bound \eqref{term2}  gives the  bound 
\begin{eqnarray*}
\lefteqn{\left| \PP (X_t  \le n) - \Phi \left( \frac{( n\mu - t) \sqrt{\mu}}{\sigma \sqrt{t}}\right) \right| } \nonumber \\
&\leq &  \frac{1}{\sqrt{t}} \max \left\{  \frac{\sqrt{2}}{e \sqrt{\pi}} \frac{\sigma}{\mu } ,   \frac{32 }{ e^2 \sqrt{2  \pi } }
  \left( \frac{\sigma^3  }{\mu^2\sqrt{t} } +  \frac{\sigma}{(224)^2 \sqrt{\mu}} \right) \right\} +  4 ( 4 \beta_2 + 3 \beta_3).
\end{eqnarray*} 
The terms $\beta_2 $ and $\beta_3$ depend on $n$ as well as on $t$ in an implicit fashion but may be straightforward to calculate in some situations. 
\end{remark}

\section{Remaining proofs of results}\label{proofs}

\begin{proof}[Proof of Lemma \ref{lemma2.4}]
To bound
$I = \left|  \Phi \left( \frac{n \mu -t }{\sigma \sqrt{n}} \right) - \Phi \left( \frac{( n\mu - t) \sqrt{\mu}}{\sigma \sqrt{t}}\right)\right| $
we consider two cases. 

{\bf Case 1: $n \mu \ge t$:}
If $t \le n \mu$ then $\frac{n \mu}{t} \ge 1$ and
\begin{eqnarray*}
I 
&\leq& \frac{1}{\sqrt{2 \pi} } \frac{n \mu -t}{\sigma \sqrt{n}}\left(  \frac{\sqrt{n \mu}}{\sqrt{t}} -1 \right)  \exp\left\{-\frac12 \left( \frac{n \mu -t}{\sigma \sqrt{n}}  \right)^2\right\}  \\
&\leq& \frac{1}{\sqrt{2 \pi} } \frac{\sigma \sqrt{n}}{ t + \sqrt{ t n \mu}} \sup_{x \ge 0} \left\{ x^2 e^{-\frac12 x^2} \right\} \\
&\leq&  \frac{\sqrt{2}}{e \sqrt{\pi } } \frac{\sigma }{  \sqrt{ t \mu}} .
\end{eqnarray*} 

\medskip 
{\bf Case 2: $t > n \mu $:}
If $t > n \mu$ then $\frac{n \mu}{t} < 1$ and
\begin{eqnarray*}
I 
&\le& \frac{1}{\sqrt{ 2 \pi}} \frac{t - n \mu }{\sigma \sqrt{n}}   \left( 1 - \frac{\sqrt{n \mu}}{ \sqrt{t}}\right) \exp\left\{-\frac12 \left(  \frac{( t - n\mu ) }{\sigma \sqrt{n}} \frac{\sqrt{n \mu}}{\sqrt{t}}  \right)^2 \right\} \\
&\le&\frac{1}{\sqrt{ 2 \pi}}
\frac{{t^2}\sigma^3 {n^{\frac{3}{2}}}}{{(n \mu)^2}{(t - n \mu)^2(\sqrt{n \mu t} + t)} }  
 \sup_{x \ge 0} \left\{ x^4 e^{-\frac12 x^2} \right\} \\
&\le & \frac{16 }{ e^2 \sqrt{2  \pi } } \frac{{t^2}\sigma^3 }{{n^{\frac{1}{2}} \mu^2}{(t - n \mu)^2(\sqrt{n \mu t} + t)} } .
\end{eqnarray*} 
This completes the proof.
\end{proof} 

\bigskip
{\bf{Proof of Theorem \ref{theorem5.1}}}

\medskip 
For the proof of Theorem \ref{theorem5.1} we first show an auxiliary result, Lemma \ref{lemma8.4explicit}, which  gives an explicit bound for Lemma 8.4 in \cite{ChGoSh}. 

Let $\xi_1,...,\xi_n$ denote independent random variables with zero means and variances summing to one. Let $W$ denote their sum, $W=\sum_{i=1}^n \xi_i$. We consider the truncated random variables and their sums
\begin{equation}
\Bar{x_i}=\xi_i\mathbbm{1}({\xi_i\leq 1}), \;\; \overline{W}=\sum_{i=1}^n \Bar{x_i}, \;\; \mbox{and} \;\; \overline{W}^{(i)}=\overline{W}-\Bar{x_i}. \label{2notation}
\end{equation}

\begin{lem}\label{lemma8.4explicit}
Let $f_z$ denote the solution to the Stein Equation \eqref{3steineqn}.
For $z>2$ and for all $s\leq t\leq 1$, we have
\begin{eqnarray} \lefteqn{\EE [(\overline{W}^{(i)}+\Bar{x_i}) f_z(\overline{W}^{(i)}+\Bar{x_i}) - (\overline{W}^{(i)}+t)f_z(\overline{W}^{(i)}+t)]} \nonumber \\
&\leq & \left(25.8+\frac{20e^{e^2-2}}{\sqrt{2\pi}}\right) e^{-\frac{z}{2}}\min(1,|s|+|t|). \label{51lem5}
\end{eqnarray}
\end{lem}

\begin{proof}[Proof of Lemma \ref{lemma8.4explicit}]
Let $g(w)=(wf_z(w))'$. Then for all $s\leq t\leq 1$,
\begin{equation}
\EE (\overline{W}^{(i)}+t) f_z(\overline{W}^{(i)}+t)-(\overline{W}^{(i)}-s) f_z(\overline{W}^{(i)}-s)] = \int_s^t \EE g(\overline{W}^{(i)}+u)du. \label{5lemma9proof1}
\end{equation}
Using \eqref{3steinsoln}, we can compute that
\begin{equation}
g(w) = \begin{cases}
 \sqrt{2\pi}(1-\Phi(z))((1+w^2)e^{w^2/2}\Phi(w)+\frac{w}{\sqrt{2\pi}})  & \mbox{if} \;w \leq z;\\
\sqrt{2\pi}\Phi(z)((1+w^2)e^{w^2/2}(1-\Phi(w))-\frac{w}
{\sqrt{2\pi}}) & \mbox{if} \;w>z.\\
       \end{cases} \label{5gw0}
\end{equation}
Instead of $w\leq z$, we consider whether or not $w\leq \frac{z}{2}$. We split the problem into four cases.

Case $1$. If $w\leq 0$, then $(5.4)$ from \cite{chenshao}  gives
\begin{equation}
\sqrt{2\pi}(1+w^2)e^{w^2/2}\Phi(w)+w \leq \frac{2}{1+|w|^3} \;\;\; \mbox{for}\; w\leq 0. \label{5ChenShao2011} 
\end{equation}
In this case, $w\leq 0<z$, so
\begin{equation}
g(w)\leq (1-\Phi(z))\frac{2}{1+|w|^3}\leq \frac{4(1+z^2)(1+z^3)}{1+|w|^3} e^{\frac{z^2}{8}}(1-\Phi(z)). \label{5gw1}
\end{equation}

Case $2$. If $0<w\leq \frac{z}{2}$, then
\begin{align}
g(w)&\leq (1-\Phi(z))(3(1+z^2)e^{\frac{z^2}{8}}+z) \nonumber\\
&\leq \frac{4(1+z^2)(1+z^3)}{1+|w|^3} e^{\frac{z^2}{8}}(1-\Phi(z)).\label{5gw2}
\end{align}

Case $3$. If $\frac{z}{2}<w\leq z$, then 
\begin{align}
g(w)&\leq \sqrt{2\pi}(1-\Phi(z))((1+z^2)e^{z^2/2}+\frac{z}{\sqrt{2\pi}})\nonumber\\ 
&\leq 8(1+z^2)e^{z^2/2}(1-\Phi(z)). \label{5gw3}
\end{align}

Case $4$. If $z<w$, then  replacing $w$ by $-w$ in \eqref{5ChenShao2011} gives  $$\sqrt{2\pi}(1+w^2)e^{w^2/2}\Phi(-w)-w \leq \frac{2}{1+|w|^3}.$$
In this case, we use the standard normal tail bound \eqref{4normaltail} to obtain
\begin{align}
g(w)&\leq \Phi(z)\frac{2}{1+|w|^3} \leq 2= 8(1+z^2)e^{\frac{z^2}{2}} \frac{e^{-z^2/2}}{4(1+z^2)} \leq 8(1+z^2)e^{\frac{z^2}{2}}(1-\Phi(z)).\label{5gw4}
\end{align}

Collecting  \eqref{5gw1}, \eqref{5gw2}, \eqref{5gw3} and \eqref{5gw4}, 
\begin{equation}
g(w) \leq  \begin{cases}
\frac{4(1+z^2)(1+z^3)}{1+|w|^3} e^{\frac{z^2}{8}}(1-\Phi(z)) & \mbox{if}\; w\leq \frac{z}{2};\\
8(1+z^2)e^{\frac{z^2}{2}}(1-\Phi(z))                         & \mbox{if}\; w>\frac{z}{2}.\\
           \end{cases}    \label{5gw5}
\end{equation}

So for any $u \in [s,t]$, since $z>2$, we have
\begin{align*}
\EE g(\overline{W}^{(i)}+u)&= \EE\left[g(\overline{W}^{(i)}+u)\mathbbm{1}_{\overline{W}^{(i)}+u\leq \frac{z}{2}}\right] + \EE\left[g(\overline{W}^{(i)}+u)\mathbbm{1}_{\overline{W}^{(i)}+u> \frac{z}{2}}\right] \\
&\leq \EE\left[\frac{1}{1+|\overline{W}^{(i)}+u|^3}\right]4(1+z^2)(1+z^3)e^{\frac{z^2}{8}}(1-\Phi(z))\\ 
&\;\;\;\;\;+ 8(1+z^2)e^{\frac{z^2}{2}}(1-\Phi(z))\PP\left(\overline{W}^{(i)}+u>\frac{z}{2}\right)
\\
&\leq \EE \left[\frac{1}{1+|\overline{W}^{(i)}+u|^3}\right]4(1+z^2)(1+z^3)e^{\frac{z^2}{8}} \frac{1}{z\sqrt{2\pi}}e^{-\frac{z^2}{2}} \\
&\;\;\;\;\;+ 8(1+z^2)e^{\frac{z^2}{2}} \frac{1}{z\sqrt{2\pi}} e^{-\frac{z^2}{2}} \PP(e^{2u}e^{2\overline{W}^{(i)}}>e^z).
\end{align*}

Using Markov's Inequality, since $u\leq t\leq 1$, we obtain
\begin{align}
\EE g(\overline{W}^{(i)}+u)&\leq \frac{4(1+z^2)(1+z^3)e^{-\frac{3z^2}{8}}}{z\sqrt{2\pi}} \EE \left[\frac{1}{1+|\overline{W}^{(i)}+u|^3}\right] \nonumber \\
&+ \frac{8(1+z^2)}{z\sqrt{2\pi}}e^{2u-z} \EE [e^{2\overline{W}^{(i)}}] \nonumber\\
&\leq \frac{4(1+z^2)(1+z^3)e^{-\frac{3z^2}{8}}}{z\sqrt{2\pi}} + \frac{8(1+z^2)}{z\sqrt{2\pi}}e^2 e^{-z} e^{e^2-3}\nonumber \\
&\leq \left(25.8+\frac{20}{\sqrt{2\pi}}e^{e^2-2}\right) e^{-\frac{z}{2}}, \label{5lemma9proof2}
\end{align}
where we used Lemma 8.2 from \cite{ChGoSh} with $t=2$ and $\alpha=B=1$. 
So for $z>2$, from \eqref{5lemma9proof2} we have
\begin{align}
\int_s^t \EE g(\overline{W}^{(i)}+u)du &\leq \left(25.8+\frac{20e}{\sqrt{2\pi}}e^{e^2-3}\right) e^{-\frac{z}{2}}(t-s) \nonumber\\
&\leq \left(25.8+\frac{20e^{e^2-2}}{\sqrt{2\pi}}\right) e^{-\frac{z}{2}} (|t|+|s|). \label{5lemma9proof3}
\end{align}

The assertion follows. 
\end{proof}


\begin{proof}[Proof of Theorem \ref{theorem5.1}]
Note that it is enough to consider $z\geq 0$. To see this, replacing $W$ by $-W$ gives
\begin{equation}
|\PP(-W\leq z)-\Phi(z)|=|\PP(-W\geq z)-\Phi(-z)|=|\PP(W\leq -z)-\Phi(-z)| .   \label{5z<0}
\end{equation}

{\bf The case $\beta_2+\beta_3\geq 1$} 

We start with the case of $\beta_2+\beta_3\geq 1$. Note that 
$$ |\PP( W \le z) - \Phi (z) | = | \PP (W > z) - ( 1 - \Phi(z))| \le \PP(W>z) + 1-\Phi(z).$$

As $W$ is sum of independent random variables with zero means and variances less than or equal to one, we apply Lemma 8.1 in \cite{ChGoSh} with $B=1$ and $p=2$ to obtain
\begin{align}
\PP(W\geq z)&\leq \PP\left(\max_{1\leq i\leq n}|\xi_i|> \frac{z\vee 1}{2}\right) + e^2\left(1+\frac{z^2}{2}\right)^{-2}\nonumber\\
&\leq \sum_{i=1}^n \PP\left(|\xi_i|>\frac{z\vee 1}{4}\right) + e^2\left(1+\frac{z^2}{2}\right)^{-2} \label{5beta>1aim}.
\end{align}
To write \eqref{5beta>1aim} as a bound of the form \eqref{5theorem3}, 
we bound $e^2(1+\frac{z^2}{2})^{-2}$ by $1.867e^2(1+z)^{-2}$. 


For $1-\Phi(z)$ we apply the standard normal tail bound \eqref{4normaltail} 
and obtain
\begin{align}
|\PP(W\leq z)-\Phi(z)|
&\leq \PP(W\geq z) +|1-\Phi(z)|\nonumber\\
&\leq \sum_{i=1}^n \PP\left(|\xi_i|>\frac{z\vee 1}{4}\right) + 1.867e^2(1+z)^{-2}(\beta_2+\beta_3) \nonumber\\
&\;\;\;\; + \min\left(\frac{1}{2},\frac{1}{z\sqrt{2\pi}}\right)e^{-\frac{z^2}{2}}. \label{51.8662}
\end{align}

Now we bound the standard normal tail bound in \eqref{51.8662} by \begin{equation}
\min\left(\frac{1}{2},\frac{1}{z\sqrt{2\pi}}\right)e^{-\frac{z^2}{2}} \leq 1.176(1+z)^{-2}.   \label{51.176}
\end{equation}

Substituting \eqref{51.176} into \eqref{51.8662} gives that for $z\geq 0$,
\begin{align}
\lefteqn{|\PP(W\leq z)-\Phi(z)|} \nonumber\\
&\leq 
\sum_{i=1}^n \PP\left(|\xi_i|>\frac{z\vee 1}{4}\right) + (1.867e^2+1.176)(1+z)^{-2}(\beta_2+\beta_3)\nonumber\\
&\leq 2\sum_{i=1}^n \PP\left(|\xi_i|>\frac{1\vee |z|}{4}\right) + (1.867e^2+1.176)(1+|z|)^{-2}(\beta_2+\beta_3). \label{5nonunifboundbeta>1}   
\end{align}
Since $1.867e^2+1.176< 15$, we have proved the theorem for the case that  $\beta_2+\beta_3 \geq 1$.

\medskip 
{\bf{The case $\beta_2+\beta_3 <1$ and $z \le 2$}} 

Next, we consider the case of $\beta_2+\beta_3 <1$. We distinguish whether or not $z>2$.

If $z\in [0, 2]$, then we use the uniform bound $(3.31)$ from \cite{ChGoSh}, which states that 
\begin{equation}
   \sup_{z\in \mathbb{R}} |\PP(W\leq z)-\Phi(z)| \leq 4.1(\beta_2+\beta_3).\label{54.1}
\end{equation}
We bound $4.1$ by $37(1+|z|)^{-2}$ for $z\in [0, 2]$ because $4.1\times (1+2)^2 <37$. So we have 
\begin{equation}
|\PP(W\leq z)-\Phi(z)|\leq 2\sum_{i=1}^n \PP\left(|\xi_i|>\frac{1\vee |z|}{4}\right) + 37(1+|z|)^{-2}(\beta_2+\beta_3).   \label{536.9}
\end{equation}
Thus we have proved the theorem when $\beta_2+\beta_3 <1$ and $z\in[0,2]$.

\medskip
{\bf{The case $\beta_2+\beta_3 <1$ and $z > 2$}} 

Our remaining task is to prove the theorem when $\beta_2+\beta_3 <1$ and $z>2$.
Recall the notations  $\Bar{x_i}=\xi_i\mathbbm{1}_{\xi_i\leq 1}$, $\overline{W}=\sum_{i=1}^n \Bar{x_i}$, and $\overline{W}^{(i)}=\overline{W}-\Bar{x_i}$.
The idea is to show that $\PP(W>z)$ is close to $\PP(\overline{W}>z)$ for $z>2$. Observing that 
\begin{align}
\{W>z\}&=\{W>z, \max_{1\leq i\leq n}\xi_i >1\}\cup \{W>z, \max_{1\leq i\leq n}\xi_i \leq 1\} \nonumber\\
       &\subset \{W>z, \max_{1\leq i\leq n}\xi_i >1\}\cup\{\overline{W}>z\},\label{5Wset}
\end{align}
and $W\geq \overline{W}$, $\PP(\overline{W}>z)$ yields 
\begin{equation}
\PP(\overline{W}>z)\leq \PP(W>z)\leq \PP(\overline{W}>z)+\PP(W>z, \max_{1\leq i\leq n}\xi_i >1).  \label{5Wset2} 
\end{equation}
From Lemma $8.3$ in \cite{ChGoSh}, with $p=2$ and $z>2$,
\begin{equation}
\PP(W \geq z, \max_{1\leq i\leq n}\xi_i >1) \leq 2\sum_{i=1}^n \PP\left(|\xi_i|>\frac{z}{4}\right) + e^2\left(1+\frac{z^2}{8}\right)^{-2}\beta_2. \label{5Wset3}
\end{equation}
For a bound of type \eqref{5theorem3}, we bound $(1+\frac{z^2}{8})^{-2}$ by $4(1+z)^{-2}$. 
Thus from \eqref{5Wset2} and \eqref{5Wset3}, 
\begin{align*}
|\PP(W\geq z)-P(\overline{W}>z)| &\leq  2\sum_{i=1}^n \PP\left(|\xi_i|>\frac{z}{4}\right) + e^2\left(1+\frac{z^2}{8}\right)^{-2}\beta_2\\
&\leq  2\sum_{i=1}^n \PP\left(|\xi_i|>\frac{z}{4}\right) + 4e^2(1+z)^{-2} (\beta_2+\beta_3),
\end{align*}
where for the last inequality we used that$\beta_3 \ge 0$. 
Hence, using the triangle inequality, we have for $z>2$,
\begin{align*}
\lefteqn{
|\PP(W\leq z)-\Phi(z)|}\\
 &\leq |\PP(W\geq z)-\PP(\overline{W}>z)|+ |\PP(\overline{W}>z)-\Phi(-z)|  \\
 &\leq 2\sum_{i=1}^n P\left(|\xi_i|>\frac{z}{4}\right) + 4e^2(1+z)^{-2} (\beta_2+\beta_3)
 + |\PP(\overline{W}\leq z)-\Phi(z)|.
\end{align*}
Note that for $z>2$, we can bound $e^{-\frac{z}{2}} \le \frac{16}{e^{1.5}}(1+z)^{-2}$. 

Now we claim that for $z>2$,
\begin{equation}
    |\PP(\overline{W}\leq z)-\Phi(z)|\leq 7115 e^{-\frac{z}{2}}(\beta_2+\beta_3)\label{5e-z/2}.
\end{equation}
If \eqref{5e-z/2} holds, then for $z>2$, bounding $e^{-\frac{z}{2}} \le \frac{16}{e^{1.5}}(1+z)^{-2}$, we obtain 
\begin{align}
\lefteqn{|\PP(W\leq z)-\Phi(z)| } \nonumber \\ &\leq 
2\sum_{i=1}^n \PP\left(|\xi_i|>\frac{z}{4}\right)+ \left(4e^2+\frac{16}{e^{1.5}}\times7115\right)(1+z)^{-2}(\beta_2+\beta_3)   \nonumber \\
&\leq
2\sum_{i=1}^n \PP\left(|\xi_i|>\frac{1\vee |z|}{4}\right)+ 25431(1+|z|)^{-2}(\beta_2+\beta_3) \label{5z>2bound}
\end{align}
which proves the theorem when $\beta_2+\beta_3<1$ and $z>2$ and therefore completes the proof of Theorem $5.1$. 
So our remaining work is to prove \eqref{5e-z/2}.

\medskip
{\bf Proof of  \eqref{5e-z/2}}

To prove  \eqref{5e-z/2} we use Stein's method as well as properties of the solution $f_z$ to the Stein Equation \eqref{3steineqn}. We define the function 
\begin{equation}
\Bar{K_i}(t)=\EE[\Bar{x_i}(\mathbbm{1}_{0 \leq t \leq \Bar{x_i}}- \mathbbm{1}_{\Bar{x_i} \leq t <0})] \label{5kfunction}
\end{equation}
where $\Bar{x_i}=\xi_i \mathbbm{1}_{\xi_i\leq 1}$.
Equation (8.24) in \cite{ChGoSh} and
$\sum_{i=1}^n \EE{\xi_i}^2=1$ give 
\begin{equation}
\sum_{i=1}^n \int_{-\infty}^1 \Bar{K_i}(t) dt= \sum_{i=1}^n \EE\Bar{x_i}^2 = 1- \sum_{i=1}^n \EE[{\xi_i}^2 \mathbbm{1}_{\xi_i>1}].   \label{5kfunction2}
\end{equation}

Using the independence between $\overline{W}^{(i)}$ and $\Bar{x_i}$, 
\begin{align}
\EE\left[\overline{W}f_z(\overline{W})\right] &= \sum_{i=1}^n \EE\left[\Bar{x_i}f_z(\overline{W})\right]\nonumber\\
&=\sum_{i=1}^n \EE[\Bar{x_i}(f_z(\overline{W})-f_z(\overline{W}^{(i)}) )]+ \sum_{i=1}^n 
\EE\Bar{x_i} \EE[f_z(\overline{W}^{(i)})]\nonumber\\
&=\sum_{i=1}^n \EE\left[\Bar{x_i} \int_0^{\Bar{x_i}} f'_z(\overline{W}^{(i)}+t) dt\right] + \sum_{i=1}^n \EE\Bar{x_i} \EE[ f_z(\overline{W}^{(i)})] \label{52p33}.
\end{align}
The first term in \eqref{52p33} can be written as
\begin{eqnarray*}
\lefteqn{\sum_{i=1}^n \EE\left[\Bar{x_i} \int_0^{\Bar{x_i}} f'_z(\overline{W}^{(i)}+t) dt\right] }\\
&=& \sum_{i=1}^n \EE\int_{-\infty}^1 f'_z(\overline{W}^{(i)}+t) \Bar{x_i} \mathbbm{1}_{0 \leq t \leq \Bar{x_i}} dt - \sum_{i=1}^n \EE\int_{-\infty}^1 f'_z(\overline{W}^{(i)}+t) \Bar{x_i} \mathbbm{1}_{\Bar{x_i}\leq t<0}dt\\
&= &\sum_{i=1}^n \int_{-\infty}^1 \EE[f'_z(\overline{W}^{(i)}+t)] \; \Bar{K_i}(t) dt
\end{eqnarray*}
where the last equality follows from independence. Therefore, 
\begin{equation}
\EE[\overline{W}f_z(\overline{W})] = \sum_{i=1}^n \int_{-\infty}^1 \EE[f'_z(\overline{W}^{(i)}+t)] \; \Bar{K_i}(t) dt + \sum_{i=1}^n \EE\Bar{x_i} \EE [f_z(\overline{W}^{(i)})]. \label{5steineqn}
\end{equation}
Next we replace $w$ by $\overline{W}$ and take expectations in the Stein Equation \eqref{3steineqn}, 
together with \eqref{5kfunction2} and \eqref{5steineqn},  to obtain
\begin{align}
\PP(\overline{W}\leq z)-\Phi(z)&= \EE[f'_z(\overline{W})] - \EE[\overline{W}f_z(\overline{W})] \nonumber \\
&=\EE[f'_z(\overline{W})]\left(\sum_{i=1}^n \int_{-\infty}^1 \Bar{K_i}(t) dt+\sum_{i=1}^n \EE[{\xi_i}^2 \mathbbm{1}_{\xi_i>1}]\right) \nonumber \\
&= \sum_{i=1}^n \EE[{\xi_i}^2 \mathbbm{1}_{\xi_i>1}] E[f'_z(\overline{W})] \label{5R1}\\
&\;\;\;\;\;+ \sum_{i=1}^n \int_{-\infty}^1 \EE[f'_z(\overline{W}^{(i)}+\Bar{x_i}) - f'_z(\overline{W}^{(i)}+t)] \Bar{K_i}(t) dt \label{5R2}\\
&\;\;\;\;\;+ \sum_{i=1}^n \EE[\xi_i \mathbbm{1}_{\xi_i> 1}] E[ f_z(\overline{W}^{(i)})] \label{5R3}\\
&= R_1 + R_2 + R_3. \nonumber
\end{align}

In order to prove \eqref{5e-z/2}, we bound each of $R_1$ given in \eqref{5R1}, $R_2$ given in \eqref{5R2} and $R_3$ given in \eqref{5R3} and show that the sum of the three bounds is less than or equal to $7115e^{-\frac{z}{2}}(\beta_2+\beta_3)$ for $z>2$.

\medskip
{\it Bound for $R_1$} 

For $R_1= \sum_{i=1}^n \EE[{\xi_i}^2 \mathbbm{1}_{\xi_i>1}] \EE[f'_z(\overline{W})]$, substituting \eqref{3steinsoln} into the Stein Equation \eqref{3steineqn} gives
\begin{align*}
f_z'(w)&=w f_z(w) + \mathbbm{1}_{w \leq z}-\Phi(z)    \\
&=\begin{cases}
(\sqrt{2\pi}we^{w^2/2}\Phi(w)+1)(1-\Phi(z)) 
& \mbox{if} \;w \leq z;\\
(\sqrt{2\pi}we^{w^2/2}(1-\Phi(w))-1)\Phi(z)  & \mbox{if}\; w>z.
\end{cases}
\end{align*}
Using \eqref{3lemma43},
\begin{align*}
\EE|f'_z(\overline{W})|&= \EE[|f'_z(\overline{W})|\mathbbm{1}_{\overline{W}\leq  \frac{z}{2}}] + \EE[|f'_z(\overline{W})|\mathbbm{1}_{\overline{W}>\frac{z}{2}}]\\
&= \EE[(\sqrt{2\pi}we^{w^2/2}\Phi(w)+1)(1-\Phi(z))\mathbbm{1}_{\overline{W}\leq \frac{z}{2}})]+ \EE[|f'_z(\overline{W})| \mathbbm{1}_{\overline{W} >\frac{z}{2}}]\\
&\leq
\left(\sqrt{2\pi}\; \frac{z}{2} e^{z^2/8}+1\right)(1-\Phi(z)) + \PP\left(\overline{W} >\frac{z}{2}\right).
\end{align*}
By Markov's inequality, 
$
\PP\left(\overline{W} >\frac{z}{2}\right)=\PP \left(e^{\overline{W}} >e^{\frac{z}{2}}\right) \leq e^{-\frac{z}{2}}\EE[e^{\overline{W}}].  $
By definition $\Bar{x_i}\leq 1$, so $\EE\Bar{x_i}\leq 0$ and $\sum_{i=1}^n \EE{\Bar{x_i}}^2\leq 1$. Applying Lemma 8.2 in \cite{ChGoSh} with $\alpha=B=t=1$ gives $$\EE[e^{\overline{W}}]\leq \exp(e-1-1)=e^{e-2}.$$
Again employing the standard normal tail bound \eqref{4normaltail}, 
\begin{align*}
\EE|f'_z(\overline{W})| 
&\le \frac{1}{2}e^{-\frac{3}{8}z^2} + \frac{1}{z\sqrt{2\pi}}e^{-\frac{z^2}{2}} + e^{-\frac{z}{2}}e^{e-2}\\
&\le \frac{1}{2}e^{-\frac{1}{2}}e^{-\frac{z}{2}}  +
\frac{e^{-1}}{2\sqrt{2\pi}} e^{-\frac{z}{2}}
 + e^{-\frac{z}{2}}e^{e-2}.
\end{align*}
Hence, we have shown that 
\begin{align}
|R_1| 
&\leq \left(\frac{1}{2}e^{-\frac{1}{2}}+\frac{e^{-1}}{2\sqrt{2\pi}}+e^{e-2}\right) e^{-\frac{z}{2}}  \sum_{i=1}^n \EE[{\xi_i}^2 \mathbbm{1}_{\xi_i>1}]\nonumber\\
&\leq \left( \frac{1}{2}e^{-\frac{1}{2}}+\frac{e^{-1}}{2\sqrt{2\pi}}+e^{e-2}\right) e^{-\frac{z}{2}} (\beta_2+\beta_3). \label{5boundR1}
\end{align}

\medskip 
{\it Bound for $R_2$}

For $R_2= \sum_{i=1}^n \int_{-\infty}^1  \EE[f'_z(\overline{W}^{(i)}+\Bar{x_i}) - f'_z(\overline{W}^{(i)}+t)] \Bar{K_i}(t) dt$, we use the Stein Equation \eqref{3steineqn} to write $R_2$ as the sum of two quantities, and then bound them separately;

\begin{align}
R_2 &= \sum_{i=1}^n \int_{-\infty}^1  \EE[(\overline{W}^{(i)}+\Bar{x_i}) f_z(\overline{W}^{(i)}+\Bar{x_i}) + \mathbbm{1}_{\overline{W}^{(i)}+\Bar{x_i} \leq z} - \Phi(z) \nonumber\\
&\;\;\;\;\;- (\overline{W}^{(i)}+t)f_z(\overline{W}^{(i)}+t) - \mathbbm{1}_{\overline{W}^{(i)}+t \leq z}+\Phi(z) ] \Bar{K_i}(t) dt \nonumber\\
&= R_{21}+R_{22} \label{5R21+R22}
\end{align}
with 
\begin{align*}
R_{21} &= \sum_{i=1}^n \int_{-\infty}^1  \EE[\mathbbm{1}_{\overline{W}^{(i)}+\Bar{x_i} \leq z}- \mathbbm{1}_{\overline{W}^{(i)}+t \leq z}]\Bar{K_i}(t) dt ; \\
R_{22} &= \sum_{i=1}^n \int_{-\infty}^1  \EE[(\overline{W}^{(i)}+\Bar{x_i}) f_z(\overline{W}^{(i)}+\Bar{x_i})-(\overline{W}^{(i)}+t)f_z(\overline{W}^{(i)}+t)]\Bar{K_i}(t) dt.
\end{align*} 

Since the difference between two indicator functions is always less than or equal to one, $R_{21}$ can be bounded by
\begin{align*}
R_{21} 
&\leq \sum_{i=1}^n \int_{-\infty}^1 \EE[\mathbbm{1}_{\Bar{x_i} \leq t} P(z-t<\overline{W}^{(i)} \leq z-\Bar{x_i}|\Bar{x_i})] \Bar{K_i}(t) dt.
\end{align*}
Applying Proposition 8.1 from \cite{ChGoSh}  with $a=z-t$ and $b=z-\Bar{x_i}$ gives
\begin{align}
R_{21} &\leq \sum_{i=1}^n \int_{-\infty}^1 \EE[6(\min(1, t- \Bar{x_i})+\beta_2+\beta_3)e^{-\frac{z-t}{2}}] \Bar{K_i}(t) dt\nonumber\\
&\leq
6e
^{-\frac{z}{2}}e^{\frac{1}{2}}  \sum_{i=1}^n \int_{-\infty}^1 \EE[\min(1, |t|+|\Bar{x_i}|)+\beta_2+\beta_3] \Bar{K_i}(t) dt \nonumber \\
&\le 6e
^{-\frac{z}{2}}e^{\frac{1}{2}} (\beta_2 + \beta_3) + 
6e
^{-\frac{z}{2}}e^{\frac{1}{2}}  \sum_{i=1}^n \int_{-\infty}^1 \EE[\min(1, |t|+|\Bar{x_i}|)] \Bar{K_i}(t) dt ,
\label{5R211}
\end{align}
where we used 
\eqref{5kfunction2} for the last step. 
Note that $\mathbbm{1}_{0 \leq t\leq \Bar{x_i}}+\mathbbm{1}_{\Bar{x_i} \leq t<0} \leq \mathbbm{1}_{|t|\leq |\Bar{x_i}|}$, so $\Bar{K_i}(t) \leq \EE[|\Bar{x_i}| \mathbbm{1}_{|t|\leq |\Bar{x_i}|}]$. Moreover, as both $\min(1, |t|+|\Bar{x_i}|)$ and $|\Bar{x_i}| \mathbbm{1}_{|t|\leq |\Bar{x_i}|}$ are increasing functions of $|\Bar{x_i}|$, they are positively correlated. So
\begin{align*}
\EE[\min(1, |t|+|\Bar{x_i}|)] \Bar{K_i}(t) &\leq  \EE[\min(1, |t|+|\Bar{x_i}|)]  \EE[|\Bar{x_i}| \mathbbm{1}_{|t|\leq |\Bar{x_i}|}]\\
&\leq \EE[\min(1, |t|+|\Bar{x_i}|) |\Bar{x_i}| \mathbbm{1}_{|t|\leq |\Bar{x_i}|}] \\
&\leq 2\EE[\min(1, |\Bar{x_i}|) |\Bar{x_i}| \mathbbm{1}_{|t|\leq |\Bar{x_i}|}].
\end{align*}
This gives 
\begin{align}
\lefteqn{
\sum_{i=1}^n \int_{-\infty}^1 \EE[\min(1, |t|+|\Bar{x_i}|)] \Bar{K_i}(t) dt } \nonumber \\ &\leq \sum_{i=1}^n \int_{-\infty}^1 2\EE[\min(1, |\Bar{x_i}|) |\Bar{x_i}| \mathbbm{1}_{|t|\leq |\Bar{x_i}|}] dt  \\
&\leq 4\sum_{i=1}^n \EE[\min(1, |\Bar{\xi_i}|) |\Bar{\xi_i}|^2] 
= 4(\beta_2+\beta_3). \label{5R213}
\end{align} 
Substituting \eqref{5R213} into \eqref{5R211},
\begin{align}
R_{21}\leq 6e^{-\frac{z}{2}}e^{\frac{1}{2}} (4+1) (\beta_2+\beta_3) = 30e^{\frac{1}{2}}(\beta_2+\beta_3)e^{-\frac{z}{2}}. \label{5R21upperbound}
\end{align}
Similarly, we can construct a lower bound for $R_{21}$ by symmetry,
\begin{align}
R_{21} 
&\geq \sum_{i=1}^n \int_{-\infty}^1 \EE[-\mathbbm{1}_{t \leq \Bar{x_i}} P(z-\Bar{x_i}<\overline{W}^{(i)} \leq z-t |\Bar{x_i})] \Bar{K_i}(t) dt \nonumber\\
&\geq - \sum_{i=1}^n \int_{-\infty}^1 \EE[6(\min(1, \Bar{x_i}-t)+\beta_2+\beta_3) e^{-\frac{z-\Bar{x_i}}{2}}] \Bar{K_i}(t) dt\nonumber\\
&\geq - 6e^{-\frac{z}{2}}e^{\frac{1}{2}}  \sum_{i=1}^n \int_{-\infty}^1 \EE[\min(1, |t|+|\Bar{x_i}|)+\beta_2+\beta_3] \Bar{K_i}(t) dt. \label{5R21lowerbound}
\end{align}
Proceeding now as for \eqref{5R213} gives that $R_{21}\geq -30e^{\frac{1}{2}}(\beta_2+\beta_3)e^{-\frac{z}{2}}$ and therefore,
\begin{equation}
    |R_{21}|\leq 30e^{\frac{1}{2}}(\beta_2+\beta_3)e^{-\frac{z}{2}}.
    \label{5R21bound}
\end{equation}

For $R_{22}$, since $wf_z(w)$ is increasing in $w$, Lemma \ref{lemma8.4explicit} gives 
\begin{align}
R_{22}
&\leq \sum_{i=1}^n \int_{-\infty}^1 \EE[\mathbbm{1}_{t\leq \Bar{x_i}}(\overline{W}^{(i)}+\Bar{x_i}) f_z(\overline{W}^{(i)}+\Bar{x_i})|\Bar{x_i} \nonumber \\
& \quad \quad \quad \quad \quad  -(\overline{W}^{(i)}+t)f_z(\overline{W}^{(i)}+t)]\Bar{K_i}(t) dt  \nonumber \\
&\leq \left(25.8+\frac{20e^{e^2-2}}{\sqrt{2\pi}}\right) e^{-\frac{z}{2}} \sum_{i=1}^n \int_{-\infty}^1 \EE[\min(1, |\Bar{x_i}|+|t|)]\Bar{K_i}(t) dt  \nonumber\\
&\leq \left(103.2+\frac{80}{\sqrt{2\pi}}e^{e^2-2}\right) e^{-\frac{z}{2}}(\beta_2+\beta_3). \label{5R221} .
\end{align}
Here we used 
 \eqref{5R213} for the last step. 
A lower bound for $R_{22}$ follows similarly,
\begin{align}
R_{22}&\geq \sum_{i=1}^n \int_{-\infty}^1  \EE[\mathbbm{1}_{\Bar{x_i}\leq t} (\overline{W}^{(i)}+\Bar{x_i}) f_z(\overline{W}^{(i)}+\Bar{x_i})|\Bar{x_i}\nonumber \\
& \quad \quad \quad \quad \quad  -(\overline{W}^{(i)}+t)f_z(\overline{W}^{(i)}+t)]\Bar{K_i}(t) dt \nonumber\\
&\geq -\sum_{i=1}^n \int_{-\infty}^1 
 \EE[\mathbbm{1}_{\Bar{x_i}\leq t} (\overline{W}^{(i)}+t)f_z(\overline{W}^{(i)}+t) \nonumber \\
& \quad \quad \quad \quad \quad -(\overline{W}^{(i)}+\Bar{x_i}) f_z(\overline{W}^{(i)}+\Bar{x_i})|\Bar{x_i}]\Bar{K_i}(t) dt \nonumber\\
&\geq -\left(103.2+\frac{80}{\sqrt{2\pi}}e^{e^2-2}\right) e^{-\frac{z}{2}}(\beta_2+\beta_3). \label{5R222}
\end{align}
Collecting \eqref{5R21bound}, \eqref{5R221} and \eqref{5R222} gives
\begin{equation}
|R_2|\leq |R_{21}|+|R_{22}|\leq \left(30e^\frac{1}{2}+103.2+\frac{80}{\sqrt{2\pi}}e^{e^2-2}\right)
e^{-\frac{z}{2}}(\beta_2+\beta_3). \label{5boundR2}
\end{equation}

\medskip
{\it Bound for $R_3$} 

Finally, for $R_3=\sum_{i=1}^n  \EE[\xi_i \mathbbm{1}_{\xi_i> 1}]  \EE[ f_z(\overline{W}^{(i)})]$, we use similar arguments as for $R_1$;
\begin{align*}
 \EE|f_z(\overline{W}^{(i)})|&=  \EE\left[|f_z(\overline{W}^{(i)})|\mathbbm{1}_{\overline{W}^{(i)}\leq \frac{z}{2}}\right] +  \EE\left[|f_z(\overline{W}^{(i)})|\mathbbm{1}_{\overline{W}^{(i)}>\frac{z}{2}}\right]\\
&\leq \sqrt{2\pi} \; e^{\frac{z^2}{8}} (1-\Phi(z)) +  \EE\left[|f_z(\overline{W}^{(i)})| \mathbbm{1}_{\overline{W}^{(i)} > \frac{z}{2}}\right].
\end{align*}
From  \eqref{3lemma45}, $0<f_z\leq \min(\frac{\sqrt{2\pi}}{4},\frac{1}{|z|}) =\frac{1}{|z|}\leq \frac{1}{2}$ for $z>2$. The standard normal tail bound \eqref{4normaltail} and Lemma $8.2$ in \cite{ChGoSh} with $\alpha=B=t=1$ give
\begin{align*}
 \EE|f_z(\overline{W}^{(i)})|&\leq \sqrt{2\pi}e^{\frac{z^2}{8}}\frac{1}{z\sqrt{2\pi}} e^{-\frac{z^2}{2}} + \frac{1}{2} P\left(\overline{W}^{(i)} >\frac{z}{2}\right)  \\
&\leq \frac{1}{z}e^{-\frac{3}{8}z^2} + \frac{1}{2} e^{-\frac{z}{2}}e^{e-2}\\
&\leq \frac{1}{2} e^{-\frac{1}{2}} e^{-\frac{z}{2}}+ \frac{1}{2} e^{-\frac{z}{2}}e^{e-2} .
\end{align*}
Hence, we have shown that
\begin{align}
|R_3
&\leq \frac{1}{2}(e^{-\frac{1}{2}}+e^{e-2})e^{-\frac{z}{2}} \sum_{i=1}^n\EE[\xi_i \mathbbm{1}_{\xi_i> 1}] \nonumber\\
&\leq \frac{1}{2}(e^{-\frac{1}{2}}+e^{e-2}) e^{-\frac{z}{2}}(\beta_2+\beta_3).\label{5boundR3}
\end{align}
Applying \eqref{5boundR1}, \eqref{5boundR2} and \eqref{5boundR3} to \eqref{5R1}, \eqref{5R2} and \eqref{5R3} respectively, we have
\begin{align}
\lefteqn{
|P(\overline{W}\leq z)-\Phi(z)|} \nonumber \\
&\leq \left(31e^{-\frac{1}{2}}+\frac{3}{2}e^{e-2}+103.2+\frac{0.5e^{-1} +80e^{e^2-2}} {\sqrt{2\pi}}\right)e^{-\frac{z}{2}}(\beta_2+\beta_3) \nonumber\\
&\leq 7115e^{-\frac{z}{2}}(\beta_2+\beta_3) \label{5e-z/22}.
\end{align}
This completes the proof of \eqref{5e-z/2} and therefore the proof of \eqref{5z>2bound}. 
Thus we have proved the theorem when $\beta_2+\beta_3<1$ and $z>2$.
Hence we have proved Theorem \ref{theorem5.1}. 
\end{proof}

\appendix

\section*{Appendix: Results from \cite{ChGoSh}}

For convenience here we give some results from \cite{ChGoSh} which we use in this paper.

Let $\xi_1,...,\xi_n$ denote independent random variables with zero means and variances summing to one. Let $W$ denote their sum, $W=\sum_{i=1}^n \xi_i$. We consider the truncated random variables and their sums
\begin{equation}
\Bar{x_i}=\xi_i\mathbbm{1}({\xi_i\leq 1}), \;\; \overline{W}=\sum_{i=1}^n \Bar{x_i}, \;\; \mbox{and} \;\; \overline{W}^{(i)}=\overline{W}-\Bar{x_i}. \label{2notation}
\end{equation}

In \cite{ChGoSh}, Proposition 8.1, it is shown that for all real $a<b$ and $i=1,...,n$, 
\begin{equation*}\label{prop81} 
   \PP(a \leq \overline{W}^{(i)} \leq b) \leq 6(\min(1,b-a)+ \beta_2 +\beta_3)e^{-\frac{a}{2}} .
\end{equation*}
Moreover, Lemmas 8.1 and 8.2 in \cite{ChGoSh} gives the next result.

\begin{lem}\label{8lemmas}[Lemmas 8.1 and 8.2] 
Let $\eta_1,...,\eta_n$ be independent random variables satisfying $E\eta_i \leq0$ for $1\leq i \leq n$ and $\sum_{i=1}^n E\eta_i^2 \leq B^2$. Then for $x>0$ and $p\geq 1$, with $S_n=\sum_{i=1}^n \eta_i$,
\begin{equation*}\label{lemma8.1} 
\PP(S_n \geq x) \leq P\left(\max_{1\leq i\leq n}\eta_i > \frac{x\vee B}{p}\right) + e^p\left(1+\frac{x^2}{pB^2}\right) ^{-p} .
\end{equation*}
If moreover, for some $\alpha>0$,  $\eta_i \leq \alpha$ for all $1\leq i\leq n$, then, for  $t>0$, 
\begin{equation*}\label{lemma8.2}
    \EE e^{tS_n} \leq \exp(\alpha^{-2}(e^{t\alpha}-1-t\alpha) B^2).
\end{equation*}
\end{lem}

Lemma 8.3 in \cite{ChGoSh} with $t=1$ gives the next result.

\begin{lem} \label{8lemmas}[Lemmas 8.3] 
Let $\xi_1,...,\xi_n$ be independent random variables with zero means and variances summing to one. Let $W=\sum_{i=1}^n \xi_i$ and $\beta_2$ be given as above. Then for $z \geq 2$ and $p \geq 2$, 
\begin{equation*}
\PP (W \geq z, \max_{1\leq i\leq n}\xi_i >1) \leq 2\sum_{i=1}^n \PP \left(|\xi_i|>\frac{z}{2p}\right) + e^p\left(1+\frac{z^2}{4p}\right)^{-p}\beta_2.
\end{equation*}
\end{lem}


\bigskip
{\bf Acknowledgement.} 
The authors would like to thank Aihua Xia (Melbourne) for helpful discussion.

%
%
%
%

\end{document}